\documentclass[final,10pt]{amsart}
\usepackage{amsmath,amsfonts,amssymb,amscd,verbatim}
\usepackage[margin=0.4in]{geometry}
\usepackage{fullpage}
\usepackage{enumerate}

\usepackage{bm}

\numberwithin{equation}{section}
\theoremstyle{plain}
\newtheorem{thm}{Theorem} 
\newtheorem{theorem}[equation]{Theorem} 

\newtheorem{corollary}[equation]{Corollary} 
\newtheorem{lemma}[equation]{Lemma}
\newtheorem{proposition}[equation]{Proposition}

\theoremstyle{definition}
\newtheorem{definition}[equation]{Definition} 
\newtheorem{notation}[equation]{Notation} 
\newtheorem{example}[equation]{Example}
 
\newtheorem{question}[equation]{Question}
\newtheorem{remark}[equation]{Remark}

\DeclareMathOperator\Aut{Aut}
\DeclareMathOperator\depth{depth}
\DeclareMathOperator\End{End}
\DeclareMathOperator\gr{gr}
\DeclareMathOperator\id{id}
\DeclareMathOperator\Inn{Inn}
\DeclareMathOperator\lm{lm}
\DeclareMathOperator\pd{pd}
\DeclareMathOperator\rank{rank}
\DeclareMathOperator\reg{reg}
\DeclareMathOperator\tr{tr}

\newcommand\CC{\mathbb C}
\newcommand\NN{\mathbb N}

\newcommand\ZZ{\mathbb Z}

\newcommand\cB{\mathcal B}

\newcommand\cS{\mathcal S}

\renewcommand{\int}{\mathrm{int}}
\newcommand\inv{^{-1}}
\newcommand{\im}{\ensuremath{\operatorname{im}}}
\newcommand\iso{\cong}
\newcommand\kk{\Bbbk}
\newcommand\tensor{\otimes}
\newcommand{\rep}{\ensuremath{\operatorname{rep}}}
\newcommand{\xra}{\xrightarrow}

\newcommand\restrict[1]{\raisebox{-.3ex}{$|$}_{#1}}

\newcommand{\GL}{\ensuremath{\operatorname{GL}}}

\renewcommand{\to}{\ensuremath{\longrightarrow}}
\newcommand{\from}{\ensuremath{\longleftarrow}}

\title{On the discriminant of twisted tensor products}
\author{Jason Gaddis, Ellen Kirkman, W. Frank Moore}
\address{Wake Forest University, Department of Mathematics and Statistics, P. O. Box 7388, Winston-Salem, NC 27109} 
\email{gaddisjd@wfu.edu, kirkman@wfu.edu, moorewf@wfu.edu}

\subjclass[2010]{16W20,11R29,16S36,16S35}
\keywords{Discriminant, Twisted tensor product, Ore extension, Skew group ring, Automorphism group}

\begin{document}

\begin{abstract}
We provide formulas for computing the discriminant of noncommutative algebras 
over central subalgebras in the case of Ore extensions and skew group extensions.
The formulas  follow from a more general result regarding 
the discriminants of certain twisted tensor products.
We employ our formulas to compute automorphism groups
for examples in each case.
\end{abstract}

\maketitle

\section{Introduction}

Throughout $\kk$ is 
an algebraically closed, characteristic zero field
and all algebras are $\kk$-algebras.
All unadorned tensor products should be regarded as over $\kk$.
Given an algebra $R$, we denote by $R^\times$ the set 
of units in $R$.
If $\sigma \in \Aut(R)$, then $R^\sigma$ denotes the
subalgebra of elements of $R$ that are fixed under $\sigma$.
We denote the center of $R$ by $C(R)$.

Automorphism groups of commutative and noncommutative algebras 
can be notoriously difficult to compute.
For example, $\Aut(\kk[x,y,z])$ is not yet fully understood.
In \cite{CPWZ1}, the authors give a method for determining the 
automorphism groups of noncommutative algebras using the discriminant.
This was studied further in \cite{CPWZ2,CPWZ3,CYZ2,CYZ1}.
Discriminants of deformations of polynomial rings
were computed using Poisson geometry in \cite{LY,NTY}.

We refer the reader to \cite{CPWZ1} for the
general definitions of {\sf trace} and {\sf discriminant}
in the context of noncommutative algebras.
We review the definitions only in the case that
$B$ is an algebra finitely generated free over 
a central subalgebra $R \subseteq C(B)$ of rank $n$.

Left multiplication defines a natural embedding
$\lm:B \rightarrow \End_C(B) \iso M_n(R)$.
The usual matrix trace defines a map 
$\tr_\int:M_n(R) \rightarrow R$
called the {\sf internal trace}.
The {\sf regular trace} is defined as the composition
$\tr_{\reg}:B \xrightarrow{\lm} M_n(R) \xrightarrow{\tr_\int} R$.
For our purposes, $\tr$ will be $\tr_{\reg}$.

Let $\omega$ be a fixed integer and 
$Z:=\{z_i\}_{i=1}^\omega$ a subset of $B$.
The {\sf discriminant of $Z$} is defined to be
\[ d_\omega(Z) = \det(\tr(z_iz_j))_{\omega \times \omega} \in R.\]
If $Z$ is an $R$-basis of $B$, then the
{\sf discriminant of $B$ over $R$} is defined to be
\[ d(B/R) =_{R^\times} d_\omega(Z),\]
where $x =_{R^\times} y$ means $x = cy$ for some $c \in R^\times$.

The discriminant is independent of $R$-linear bases of $B$ \cite[Proposition 1.4]{CPWZ1}.
Moreover, if $\phi \in \Aut(B)$ and $\phi$ preserves $R$,
then $\phi$ preserves the ideal 
generated by $d(B/R)$ \cite[Lemma 1.8]{CPWZ1}.

Computing the discriminant is a computationally difficult task,
even for algebras with few generators.
For example, the matrix obtained from $\tr(z_iz_j)$ for the 
skew group algebra $\kk_{-1}[x_1,x_2,x_3] \# \cS_3$
has size $288 \times 288$.
Our first goal is to provide methods for obtaining the discriminant 
in cases where the algebra may be realized as an extension of
a smaller algebra where computations may be easier.

If $A$ is an algebra and $\sigma \in \Aut(A)$,
then the {\sf Ore extension} $A[t;\sigma]$ is generated by $A$
and $t$ with the rule $ta=\sigma(a)t$ for all $a \in A$.

\begin{thm}[Theorem \ref{thm.ore}]
Let $A$ be an algebra and set $S=A[t;\sigma]$,
where $\sigma \in \Aut(A)$ has order $m<\infty$ and no
$\sigma^i$, $1 \leq i < m$, is inner.
Suppose $R$ is a central subalgebra of $S$ and set $B = R \cap A^\sigma$.
If $A$ is finitely generated free over $B$ of rank $n$ and $R=B[t^m]$, then
$S$ is finitely generated free over $R$ and
\[ d(S/R) =_{R^\times} \left(d(A/B)\right)^m \left(t^{m-1}\right)^{mn}.\]
\end{thm}

We say an automorphism $\sigma$ of $A$ is {\sf inner} if there exists
$a \in A$ such that $xa = a\sigma(x)$ for all $x \in A$.
This is not the standard definition of an inner automorphism
but it agrees if $a$ is a unit because then $a\inv ba = \sigma(b)$.
Since $a$ is normal and assuming $A$ is a domain,
one can localize at the Ore set of powers of $a$
so as to get an inner automorphism in $A[a\inv]$.
When $\sigma^k$ is an inner automorphism for some
$1 \leq k < m = |\sigma|$ then the center of $A[t;\sigma]$
can be larger than $(C(A) \cap A^\sigma)[t^m]$.
We denote the set of all inner automorphisms of $A$ by $\Inn(A)$.  It is
a routine verification that $\Inn(A)$ forms a subgroup of $\Aut(A)$.

Let $A$ be an algebra and $G$ a finite group that acts on $A$ as
automorphisms. Denote by $\kk G$ the group algebra of $G$.
The {\sf skew group algebra} $A\# G$ has the underlying
set $A \tensor \kk G$ and multiplication defined by
\[ (a \tensor g)(b \tensor h) = a (g.b) \tensor gh \text{ for all }
a,b \in A, \;\; g,h \in G.\]

The natural embedding
\begin{align*}
A &\rightarrow A \# G \\
a &\mapsto a \tensor e
\end{align*}
where $e$ is the identity of $G$,
allows us to identify $A$ with its image in $A \# G$.
If $G$ contains no non-identity inner automorphisms
and acts faithfully on $A$, then by Lemma \ref{lem.center},
$C(A\# G) = C(A)^G$ under the above identification.

\begin{thm}[Theorem \ref{thm.skgrp}]
Let $A$ be an algebra and $G$ a finite group that acts on $A$ as automorphisms
such that no non-identity element of $G$ is inner.
Set $S=A\# G$ and identify $A$ with its image under
the embedding $a \mapsto a \tensor e$.
Suppose $A$ is a finitely generated free over 
the subalgebra $R \subseteq C(A)^G$. 
Then $S$ is finitely generated free over $R$ and
\[ d(S/R) =_{R^\times} d( A/R )^{|G|}.\]
\end{thm}

The condition that $A$ is a finitely generated free $A^G$-module is satisfied
in case $A$ is a commutative polynomial ring and $G$ is a group generated by reflections
by the classical results of Chevalley \cite{Chev} and Shephard-Todd \cite{ShTo}.
Section \ref{sec.refl} is devoted to showing that such discriminants may be computed
in a manner similar to the discriminant of an algebraic number field; see Proposition \ref{prop.trick}.

Both Ore extensions (by an automorphism)
and skew group algebras are examples of \emph{twisted tensor products}.
We prove a more general formula regarding discriminants of certain twisted
tensor products from which the two previous theorems follow in Section \ref{sec.twistThm}.  
The necessary background for these results is in Sections \ref{sec.twistbg} and \ref{sec.discbg}.
We then apply this result to the case of Ore extensions (Section \ref{sec.ore})
and skew group algebras (Section \ref{sec.skgrp}), 
as well as provide examples of each.
Finally, in Section \ref{sec.autos} we use Theorems 1 and 2 
to compute the automorphism groups of some Ore extensions and skew group algebras.
Many computations herein were assisted by routines in Macaulay2
using the NCAlgebra package\footnote{Available at http://users.wfu.edu/moorewf.}.

We thank the referee for a careful reading of the original manuscript, and many helpful suggestions.
\section{Twisted tensor products and monoid algebras} \label{sec.twistbg}

Let $A$ and $B$ be algebras and let 
$\tau:B \tensor A \rightarrow A \tensor B$ be a $\kk$-linear homomorphism
subject to the conditions that $\tau(b \tensor 1_A)=1_A \tensor b$ and
$\tau(1_B \tensor a) = a \tensor 1_B$, $a \in A$, $b \in B$.
A multiplication on $A \tensor B$ is then given
by $\mu_\tau := (\mu_A \tensor \mu_B) \circ (\id_A \tensor \tau \tensor \id_B)$
where $\mu_A$ and $\mu_B$ are the multiplication maps on $A$
and $B$, respectively.
By \cite[Proposition 2.3]{cap}, $\mu_\tau$ is associative if and only if 
$\tau \circ (\mu_B \tensor \mu_A) 
	= \mu_\tau \circ (\tau \tensor \tau) \circ (\id_B \tensor \tau \tensor \id_A)$
as maps $B \tensor B \tensor A \tensor A \rightarrow A \tensor B$.
The triple  $(A \tensor B,\mu_\tau)$ is
a {\sf twisted tensor product} of $A$ and $B$, 
denoted by $A \tensor_\tau B$.

We are concerned with twisted tensor products when $B$ is the monoid
algebra of a monoid $M$ that acts on $A$ as automorphisms.  
Let $M$ be a monoid and $\rho : M \to \Aut(A)$ a monoid homomorphism
(so that $\rho(mm') = \rho(m)\rho(m')$ and $\rho(e_M) = \id_A$ 
where $e_M$ denotes the identity of $M$).  
For $m \in M$ and $a \in A$, we write $ma$ for $\rho(m)(a)$.
For a monoid $M$, we let $C(M)$ denote the center of $M$.

In this case, one may check that since $\rho$ is a homomorphism, the assignment
\begin{eqnarray*}
\tau : \kk M \otimes A & \to & A \otimes \kk M \\
m \otimes a & \mapsto & ma \otimes m
\end{eqnarray*}
extends linearly to a $\kk$-linear map that makes the multiplication $\mu_\tau$ associative.
We will denote such a twisted tensor product by $A \otimes_\tau \kk M$ when the homomorphism
$\rho : M \to \Aut(A)$ is understood.  We identify the elements of $A$, $\kk M$ and $M$ with
their images in $A \otimes_\tau \kk M$ under the canonical embeddings
$$A \to A \otimes_\tau \kk M \from \kk M \from M.$$

For such an action of a monoid $M$ on an algebra $A$, we let $A^M$ denote the set
$$A^M = \{a \in A~|~ma = a~\text{for all}~m\in M\}.$$
It is easy to check that $A^M$ is a subalgebra of $A$.  Furthermore, since the center
of an algebra is preserved under any automorphism, $M$ acts on $C(A)$ as well, so one
may also consider $C(A)^M$.

\begin{remark} \label{rem.OreAndSkgpAreTwists}
The two main applications of interest are Ore extensions (by an automorphism) and
skew group algebras, and each fit into this framework.  In the case of an Ore
extension $A[t;\sigma]$ for an automorphism $\sigma$ of $A$, $A[t;\sigma] \cong A \otimes_\tau \kk\NN$
where $\rho : \NN \to \Aut(A)$ sends $1$ to $\sigma$.
Similarly, a finite group $G$ acting on $A$ as automorphisms is the same as a group homomorphism
$\rho : G \to \Aut(A)$, and one may check that $A \# G = A \otimes_\tau \kk G$.
\end{remark}

To better explain some of the hypotheses we will need in our theorem regarding discriminants
of twisted tensor products, we must discuss the center of $A \otimes_\tau \kk M$.

\begin{lemma}
\label{lem.center}
Let $A$ be an algebra, 
$M$ a monoid that acts on $A$ through the monoid homomorphism $\rho : M \to \Aut(A)$,
and $H = \ker \rho$.
Suppose that $H \subseteq C(M)$ and $\im \rho \cap \Inn(A) = \{\id_A\}$.
Then $C(A \tensor_\tau \kk M) = C(A)^M \otimes \kk H$.
\end{lemma}

\begin{proof}
Let $T=A \tensor_\tau \kk M$ and choose $\sum_{m \in M} a_m \tensor m \in C(T)$.
Then for all $x \in A$ we have
\[
	\left(\sum_{m \in M} a_m \tensor m \right)(x \tensor e_M)
		= \sum_{m \in M} a_m (mx) \tensor m.
\]
On the other hand,
\[	(x \tensor e_M)\left(\sum_{m \in M} a_m \tensor m \right)
		= \sum_{m \in M} x a_m \tensor m.
\]
Since $M$ acts as automorphisms, each nonzero $a_m$ is a normal element corresponding
to the automorphism induced by $m$.
This implies that if $a_m \neq 0$, then $m$ induces an inner automorphism of $A$, and hence by hypothesis
$m \in \ker \rho$.  Note that in this case it also follows that $a_m \in C(A)$, so that each term in the
sum $\sum_{m \in M} a_m \tensor m$ has $a_m \in C(A)$ and $m \in H$.

Now for all $m' \in M$, one has
\begin{equation*}
\left(\sum_{m \in H} a_m \tensor m \right)(1 \tensor m') = \sum_{m \in M} a_m \tensor mm'.
\end{equation*}
On the other hand, 
\begin{eqnarray*}
(1 \otimes m')\left(\sum_{m \in H} a_m \tensor m \right) & = & \sum_{m \in M} m'a_m \tensor m'm \\
   & = & \sum_{m \in M} m'a_m \tensor mm' \\
\end{eqnarray*}
where the second equality follows since $H \subseteq \ker \rho$.  Therefore we have that $m'a_m = a_m$ for
all $m' \in M$, hence $a_m \in C(A)^M$.  The claim now follows.
\end{proof}

We note that the hypothesis $H \subseteq \ker \rho$ in Lemma \ref{lem.center} is trivially satisfied
when $M$ is a group and $M$ acts faithfully on $A$ 
(as is the case in most skew group algebra computations),
as well as when $\NN$ acts on $A$ as an automorphism $\sigma$ (as in the Ore extension case).

The hypothesis $\im \rho \cap \Inn(A) = \{\id_A\}$ is a bit more restrictive however.  Indeed, extending Lemma \ref{lem.center} to the case where a non-identity 
element of $M$ acts as an inner automorphism is nontrivial in general, 
but can be done for an Ore extension of a domain.   
For an automorphism $\sigma$ of $A$ we define
\[ N(\sigma) = \{ a \in A^\sigma: x a = a \sigma(x) \text{ for all } x \in A\}.\]
Note that $N(\sigma) = \{ 0 \}$ when $\sigma$ is not an inner automorphism 
induced by an element of $A^\sigma$, $N(1_A) = A^\sigma \cap C(A)$, 
and if $|\sigma| = m$ then $N(\sigma^\ell) = N(\sigma^i)$, 
where $\ell \equiv i \mod m$.

\begin{lemma}
Let $A$ be a domain, $\sigma$ an automorphism of $A$ 
with $|\sigma| = m$, and set $S = A[t;\sigma]$.  Then $C(S) = \oplus N(\sigma^i)[t^i]$,
where $N(\sigma^i)$ is nonzero if $\sigma^i$ an inner automorphism 
induced by an element of $A^\sigma$.
\end{lemma}

\begin{proof}
The containment $\supseteq$ is clear, so 
suppose $\sum a_i t^i \in C(S)$;
then by the $t$-grading, $a_it^i \in C(S)$ for each $i$. Hence,
\[ \left(\sigma(a_i) t^i\right)t = t\left(a_i t^i\right) = \left(a_i t^i\right)t.\]
Thus $\sigma(a_i)=a_i$ for each $i$, so that $a_i \in A^\sigma$.
For $a_it^i \in C(S)$ and any $b \in A$, we have
$ba_it^i = a_it^ib = a_i\sigma^i(b)t^i$.
Because $S$ is a domain, then $ba_i=a_i\sigma^i(b)$ for all $i$, 
and $a_i \in N(\sigma^i)$.
Hence, for $i \not\equiv 0 \mod m$ we have 
that $\sigma^i$ is a non-identity inner automorphism induced by $a_i$.
\end{proof}

Before closing this section, we will need a lemma that characterizes
when an extension of a monoid algebras is a free extension.  Before stating the
lemma, we need a definition.

\begin{definition}
A monoid $M$ is {\sf left cancellative} provided for all $a,b,c \in M$, 
the equality $ab = ac$ implies that $b = c$.  
One similarly defines {\sf right cancellative} monoid, 
and $M$ is called {\sf cancellative} if $M$ is both left and right cancellative.
\end{definition}

\begin{lemma} \label{lem.basisCosets}
Let $H$ be a submonoid of a cancellative monoid $M$, 
and let $\cB = \{m_1,\dots,m_\ell\}$ be a subset of $M$.  
Then $\mathcal{B}$ is a basis of $\kk M$ as a left (respectively right) 
$\kk H$-module if and only if the right (respectively left) cosets of $H$
represented by all the elements of $\mathcal{B}$ are disjoint, 
and $M$ is the union of the right (respectively left) cosets of $H$ 
represented by all the elements of $\mathcal{B}$.
\end{lemma}

\begin{proof}
Note that $\mathcal{B}$ is a basis of $\kk M$ as a left $\kk H$-module
if and only if for each $m \in M$ there exists a unique $h \in H$ and $1 \leq i \leq \ell$
such that $m = hm_i$.  This is precisely the statement that the right cosets of $H$
by all the elements of $\mathcal{B}$ are disjoint and cover $M$.
\end{proof}

A source of examples of cancellative monoids are submonoids of groups.  
We record here a lemma that we will use later regarding extensions 
of monoid algebras in this context.

\begin{lemma} \label{lem.monBasis}
Let $M$ be a monoid which embeds in a group $G$, $K$ a group,
$\rho : G \to K$ a group homomorphism, 
and $H = \ker \rho \cap M$.
Suppose that, for all $g \in \ker \rho$, one has $g \in M$ or $g^{-1} \in M$. 
Then:
\begin{enumerate}[(a)]
\item \label{lem.monBasis.a} For all $m_1,m_2 \in M$, $\rho(m_1) =
  \rho(m_2)$ if and only if either $Hm_1 \subseteq Hm_2$ or $Hm_2
  \subseteq Hm_1$, and similarly for left cosets.
\item Suppose that $\kk M$ is a free left or right $\kk H$-module with basis
  $\{m_1,\dots,m_\ell\}$ with $m_i \in M$ and $m_1 = e_M$.  Then for
  every $j = 1,\dots,\ell$, there exists a unique $i$ such that
  $m_im_j \in H$.
\end{enumerate}
\end{lemma}

\begin{proof}
We prove each of the claims in the case of right cosets and a left module structure as in Lemma \ref{lem.basisCosets}.

It is easy to see that either containment of cosets above implies $\rho(m_1) = \rho(m_2)$.
Conversely, suppose that $\rho(m_1) = \rho(m_2)$.  Then $m_1m_2^{-1} \in \ker \rho$
so by hypothesis either $m_1m_2^{-1} \in H$ or $m_2m_1^{-1} \in H$.  In the first case, one has $Hm_1 \subseteq Hm_2$ and in the
second case, one has $Hm_2 \subseteq Hm_1$, proving part (\ref{lem.monBasis.a}).

For the second claim, Lemma \ref{lem.basisCosets} shows that given any element $m \in M$, there exists
a unique $i$ such that $m \in Hm_i$.  We will denote the assignment $m \mapsto m_i$ by the notation $\rep(m)$.
Fix $1 \leq j \leq n$  and consider the following function defined on the cosets:
\begin{align*}
\Phi_j : \{Hm_1,\dots,Hm_\ell\} &\rightarrow \{Hm_1,\dots,Hm_\ell\} \\
      Hm_i &\mapsto H\rep(m_im_j).
\end{align*}
Suppose that $\Phi_j(m_i) = \Phi_j(m_{i'})$.  Then one has that $m_im_j$ and $m_{i'}m_j$
are in the same coset $Hm_k$ for some $k$.  It follows that $\rho(m_i) = \rho(m_{i'})$ and hence
by part (\ref{lem.monBasis.a}) one has that $Hm_i$ and $Hm_{i'}$ intersect nontrivially.
Therefore by Lemma \ref{lem.basisCosets} one has $m_i = m_{i'}$.
This shows that $\Phi_j$ is one-to-one and hence onto.  Since we chose $m_1 = e_M$,
one of the cosets is $H$ hence given any $j$, there is a unique $i$ such that $m_im_j \in H$
as claimed.
\end{proof}

Note that $\NN$ is a submonoid of the group $\ZZ$ that intersects every subgroup of $\ZZ$ 
nontrivially, and hence fits into the framework of each of the two previous lemmas.  
Furthermore, if $M$ is a group
then each of the two lemmas' hypotheses hold trivially.  
Therefore, these hypotheses are not
restrictive from the point of view of our intended applications.

\section{Background on discriminants} \label{sec.discbg}

In this section, we assume that $A/B$ is a free extension of algebras with
$B \subseteq C(A)$.

\begin{notation} \label{not.xsigma}
Suppose $\sigma \in \Aut(A)$ and that $\sigma$ restricts to the identity on $B$.  
Then $\sigma : A \to A$ is a $B$-module homomorphism.  
Therefore, given a basis $\{x_1,\dots,x_g\}$ of $A$
as a $B$-module, we can represent $\sigma$ using a matrix with entries in $B$ which
we denote $X_\sigma$.  That is, if $x \in A$ and the coordinate vector of $x$ with
respect to the chosen basis is $\boldsymbol{x}$, then $\sigma(x) = X_\sigma\boldsymbol{x}$.
Note that the determinant of $X_\sigma$ is independent
of the chosen basis of $A$ over $B$, and is an element of $B^\times$.
We therefore denote $\det X_\sigma$ by $\det_{A/B} \sigma$.

If $\rho : M \to \Aut(A)$ is a monoid homomorphism and $m \in M$, we
write $X_m$ and $\det_{A/B} m$ to denote $X_{\rho(m)}$ and $\det_{A/B} \rho(m)$,
respectively.
\end{notation}

In the following definition we consider $\sigma \in \Aut(A)$ 
that restricts to the identity on $B$
to twist the standard trace pairing.
This will be necessary for our calculations that appear in Section \ref{sec.twistThm}.
Recall that for an algebra $A$, we use $\mu_A$ to denote multiplication on $A$.

\begin{definition}
\label{def.traceform}
Let $\sigma \in \Aut(A)$ such that $\sigma$ restricts to the identity on $B$.
Define the {\sf trace form of the extension $A/B$ twisted by $\sigma$}
(denoted $\tr_{A/B,\sigma}$) to be the $B$-bilinear pairing given by the
composition
\begin{equation*}
\tr_{A/B,\sigma} : A \times A \xra{\id_A \times \sigma} A \times A \xra{\mu_A} A \xra{\tr} B.
\end{equation*}
That is, $\tr_{A/B,\sigma}(y,z) = \tr(y\sigma(z))$.  In the case $\sigma = \id_A$,
we use $\tr_{A/B}$ to denote $\tr_{A/B,\id_A}$.
\end{definition}
Note that $\tr_{A/B,\sigma}(y,z) = \tr_{A/B,\sigma}(\sigma(z),\sigma^{-1}(y))$, so that
this bilinear pairing need not be symmetric for a general $\sigma$, but it is symmetric if $\sigma = \id_A$.

\begin{notation} \label{not.wsigma}
Given a basis $\{x_1,\dots,x_g\}$ of $A$ as a $B$-module,
the matrix of $\tr_{A/B,\sigma}$  with respect to this basis
is $W_\sigma = (\tr(x_i\sigma(x_j)))_{ij}$.
In this way, if $y$ and $z$ have representatives in the 
above basis given by vectors $\boldsymbol{y}$ and $\boldsymbol{z}$,
then $\tr_{A/B,\sigma}(y,z) = \boldsymbol{y}^TW_\sigma\boldsymbol{z}$.
We let $W$ denote the matrix $W_{\id_A}$.
\end{notation}

\begin{definition}[\cite{CPWZ1}]
For a free extension $A/B$, we define the {\sf discriminant of $A$ over $B$}
to be determinant of the matrix $W$ representing the trace pairing $\tr_{A/B}$
with respect to some chosen basis of $A$ over $B$, as in Notation \ref{not.wsigma}.
\end{definition}

\begin{lemma} \label{lem.wsigma}
Let $A/B$ be a free extension of algebras, let $\sigma \in \Aut(A)$
restrict to the identity on $B$, and let $\{x_1,\dots,x_g\}$ be a basis of $B$
over $A$.  Then using the same notation appearing in \ref{not.xsigma} and \ref{not.wsigma},
one has $W_\sigma = WX_\sigma$ and therefore $\det(W_\sigma) = \det(W)\det_{A/B} \sigma$.
\end{lemma}

\begin{proof}
Let $y$ and $z$ be elements of $A$, with representations $\boldsymbol{y}$ and $\boldsymbol{z}$
in the chosen basis, respectively.  Then since $\tr_{A/B,\sigma} = \tr_{A/B}\circ(\id_A \times \sigma)$
one has the following string of equalities from which the claim follows:
\begin{equation*}
\tr_{A/B,\sigma}(y,z) = \tr_{A/B}(y,\sigma(z)) = \boldsymbol{y}^TWX_\sigma\boldsymbol{z}. \qedhere
\end{equation*}
\end{proof}

\section{Discriminants and reflections} \label{sec.refl}

In this section, we collect some results from classical commutative
invariant theory that we will need for our examples.  
We prove that when $G$ is a group generated by reflections 
acting on the polynomial ring $A = \kk[x_1,\dots,x_n]$, 
the discriminant of the extension $A/A^G$ may be computed in a 
manner similar to the formula for the discriminant of an algebraic number field, c.f. \cite[Proposition 2.26]{milne}.
Recall that our field $\kk$ is of characteristic zero.

Let $\sigma \in \GL_n(\kk)$.  Recall that $\sigma$ is a {\sf reflection} 
if $\sigma$ is of finite order and fixes a codimension one subspace
of the vector space of linear forms in $A$.
Denote the quotient field of a domain $S$ by $Q(S)$.

\begin{theorem}
Let $A = \kk[x_1,\dots,x_n]$ and 
$G \subseteq \GL_n(\kk)$ a finite group generated by reflections
that acts on $A$ as automorphisms. Then:
\begin{enumerate}
 	\item The invariant ring $A^G = \kk[f_1,\dots,f_n]$ is a graded 
 	subalgebra of $A$, with the $f_i$ algebraically independent.
 	\item One has $\displaystyle\prod_i \deg(f_i) = |G|$.
 	\item $A$ is free as an $A^G$-module of rank $|G|$.
 	\item $Q(A^G) = Q(A)^G$, where the action of $G$ on $Q(A)$ 
 	is induced from the action of $G$ on $A$.
\end{enumerate}
\end{theorem}

\begin{proof}
The first claim is the Shephard-Todd-Chevalley theorem.  
The second claim is well-known, see \cite[Corollary 4.4]{S}.  
The third claim follows from a Hilbert series computation, 
and the last statement follows from considering the
Galois extension $Q(A)/Q(A)^G$.
\end{proof}

\begin{lemma}
Let $A = \kk[x_1,\dots,x_n]$ and 
$G \subseteq \GL_n(\kk)$ a finite group generated by reflections
that acts on $A$ as automorphisms.
Then for all $f \in A$, one has
\[ \tr_{A/A^G}(f) = \tr_{Q(A)/Q(A^G)}(f) 
	= \sum_{\sigma \in G} \sigma(f).\]
\end{lemma}

\begin{proof}
The extension $Q(A)/Q(A)^G = Q(A)/Q(A^G)$ is Galois, 
and hence we have that the usual trace map
\[ \tr_{Q(A)/Q(A^G)} f = \sum_{\sigma \in G} \sigma(f) \in Q(A^G) \]
may be computed by the trace of the $Q(A^G)$-linear map 
$\theta^{Q(A)}_f : Q(A) \to Q(A)$ given by multiplication by $f$.
Since for all $f \in A$, one has 
$\theta^{Q(A)}_f = \theta^{A}_f \otimes_{A^G} Q(A^G)$, 
we have the desired result.
\end{proof}

The following proposition is useful in computations involving discriminants
of extensions of commutative polynomial rings, since in practice
the matrix $W$ from Notation \ref{not.wsigma} can be time-consuming to obtain directly.
This result is reminiscent of the formula for the discriminant
of an algebraic number field $K$ in terms of the square of the determinant
of the matrix whose entries correspond to the evaluations of an integral
basis of $\mathcal{O}_K$ at the different embeddings of $K$ into $\CC$.

\begin{proposition}
\label{prop.trick}
Let $A = \kk[x_1,\dots,x_n]$ and 
$G = \{\sigma_1,\dots,\sigma_g\} \subseteq \GL_n(\kk)$ 
a finite group generated by reflections.  
Let $\{z_1,\dots,z_g\}$ be a basis of $A$ as an $A^G$-module, 
$W$ be the matrix of the trace form of the extension 
$A/A^G$ with respect to this basis, 
and let $M$ be the matrix $(\sigma_i(z_j))$.  
Then $W = M^TM$.  As a consequence, one has $d(A/A^G) = (\det M)^2$.
\end{proposition}

\begin{proof}
One has that
\begin{equation*}
(M^TM)_{ij} 
	= \sum_k \sigma_k(z_i)\sigma_k(z_j) 
	= \sum_k \sigma_k(z_iz_j) 
	= \tr(z_iz_j) = W_{ij}.
\end{equation*}
The claim regarding the discriminant of the extension $A/A^G$ follows
since $\det W = d(A/A^G)$.
\end{proof}

We record a corollary of this proposition when $G$ is generated by a single reflection
for later use.

\begin{corollary} \label{cor.refl}
Let $A = \kk[x_1,\dots,x_n]$ and let $\sigma$ be a reflection of order $m$.
Let $A^\sigma$ be the set of elements of $A$ left invariant by $\sigma$,
and let $f$ be a linear form such that $\sigma(f) = \xi f$ for some primitive
$m^\text{th}$ root of unity $\xi$.  Then $d(A/A^\sigma) =_{\kk^\times} f^{(m-1)m}$.
\end{corollary}

\begin{proof}
After a change of variable, we have $A = \kk[f,y_2,\dots,y_n]$
with $\sigma(f) = \xi f$ and $\sigma(y_i) = y_i$.  
Therefore, $A^\sigma = \kk[f^m,y_2,\dots,y_n]$
and hence a basis for $A$ over $A^\sigma$ is $\{1,\dots,f^{m-1}\}$.
The matrix $M$ from the above proposition is 
therefore a Vandermonde matrix on the elements
$\{f,\xi f, \xi^2 f, \dots, \xi^{m-1}f\}$. 
Therefore by Proposition \ref{prop.trick}
we have that
\begin{equation*}
\det W = \left(\prod_{i < j} (\xi^i f - \xi^j f)\right)^2 =_{\kk^\times} f^{2\binom{m}{2}} = f^{(m-1)m}. \qedhere
\end{equation*}
\end{proof}

\begin{corollary}
\label{cor.ztr}
Let $S$ be an algebra, $A=S[t]$, and $R=S[t^m]$, $m \in \ZZ_{>0}$.
Then $d(A/R)=_{\kk^\times} (t^{m-1})^m$.
\end{corollary}

\begin{proof}
This follows by applying Corollary \ref{cor.refl} to the automorphism $\sigma$
of $A$ defined by $\sigma(t) = \xi t$ where $\xi$ is a primitive $m^\text{th}$ root of unity.
\end{proof}

\begin{example}
\label{ex.discS3}
Let $A=\kk[x_1,x_2,x_3]$ and $G=\cS_3$,
the symmetric group acting as permutations of $x_i$.
A basis for $A$ over $A^G$ is $\{1,x_1,x_2,x_1^2,x_1x_2,x_1^2x_2\}$
(this is well-known, see e.g. \cite[Proposition V.2.20]{H}).
Let $M=(\sigma_i(z_j))$ so that
\[
M = \begin{pmatrix}
1 & x_1 & x_2 & x_1^2 & x_1x_2 & x_1^2x_2 \\
1 & x_2 & x_1 & x_2^2 & x_1x_2 & x_1x_2^2 \\
1 & x_3 & x_2 & x_3^2 & x_2x_3 & x_2x_3^2 \\
1 & x_1 & x_3 & x_1^2 & x_1x_3 & x_1^2x_3 \\
1 & x_2 & x_3 & x_2^2 & x_2x_3 & x_2^2x_3 \\
1 & x_3 & x_1 & x_3^2 & x_1x_3 & x_1x_3^2
\end{pmatrix}.
\]
It can be checked that the determinant of $M$ 
is the cube of the Vandermonde determinant on the set $\{x_1,x_2,x_3\}$
and so it follows from Proposition \ref{prop.trick} that
\[ d(A/A^G) = \left[\prod_{i < j} (x_i-x_j) \right]^6.\]
\end{example}

In light of the Example \ref{ex.discS3},
we conjecture that the following question has an affirmative answer. 

\begin{question}
Suppose $G = S_n$ acts on $A = \kk[x_1,\dots,x_n]$ as permutations.
Is the discriminant of $A$ over $A^G$ the Vandermonde determinant
on $\{x_1,\dots,x_n\}$ to the power $n!$?
\end{question}

\section{Discriminants of twisted tensor products} \label{sec.twistThm}

In this section, we prove the main theorem that allows us to calculate
discriminants for certain Ore extensions and skew group algebras.

\begin{theorem} \label{thm.twist}
Let $A$ be an algebra, $M$ a submonoid of a group $G$, and suppose $\rho : G \to \Aut(A)$ is a group homomorphism
such that $\im(\rho\restrict{M}) \cap \Inn(A) = \{\id_A\}$ and $H := \ker \rho \cap M \subseteq C(M)$.
Set $T = A \otimes_\tau \kk M$, and suppose $R$ is a central subalgebra of $T$ such that:
\begin{enumerate}[(a)]
\item \label{thm.twist.a} $A$ is free over $A \cap R$ of rank $n < \infty$,
\item \label{thm.twist.b} $R = (A \cap R) \otimes \kk H$, and
\item \label{thm.twist.c} There exists a basis $\{m_1,\dots,m_\ell\}$ of $\kk M$ over $\kk H$ with $m_1 = e_M$ 
and $m_i \in M$ (c.f. Lemma \ref{lem.basisCosets}).
\end{enumerate}
Then:
$$d(T/R) = \Big(d(A/A\cap R)\Big)^\ell~~\Big(d(\kk M/\kk H)\Big)^n.$$
\end{theorem}

\begin{remark}
By Lemma \ref{lem.center}, $C(A\otimes_\tau\kk M) = C(A)^M \otimes \kk H$.
Since $R \subseteq C(A\otimes_\tau\kk M)$, then assuming (b) in Theorem \ref{thm.twist}
one has $A \cap R = C(A)^M \cap R$.
\end{remark}

Before giving the proof of Theorem \ref{thm.twist}, we need one more lemma:

\begin{lemma} \label{lem.trace}
Under the hypotheses of Theorem \ref{thm.twist}, $T$ is free over $R$.
Furthermore, for all $a \in A$ and $m \in M$, one has
$$\tr(a \otimes m) = \begin{cases} \tr(a)\otimes\tr(m) & \text{if}~m \in H \\ 0 & \text{otherwise.}\end{cases}$$
\end{lemma}

\begin{proof}
Let $B = A \cap R = C(A)^M \cap R$, and let $\{x_1,\dots,x_n\}$ be a basis of $A$ over $B$. Then
$$\{x_i\otimes m_j~|~1 \leq i \leq n~\text{and}~1\leq j \leq \ell\}$$
is a basis of $T$ over $R$.  Recall that for a fixed $1 \leq \alpha \leq n$ and $1 \leq \beta \leq \ell$, one has
\begin{equation} \label{eq.twistprod}
(a \otimes m)(x_\alpha\otimes m_\beta) = a(mx_\alpha)\otimes mm_\beta.
\end{equation}
Consider the coefficient $c_{\alpha\beta} \in R$ of $x_\alpha \otimes m_\beta$ when writing the product in \eqref{eq.twistprod}
in the $\{x_i\otimes m_j\}$ basis.  If $c_{\alpha\beta} \neq 0$, then hypothesis (\ref{thm.twist.b}) implies
that $c_{\alpha\beta}$ would have a summand of the form $a' \otimes m'$ for some $m' \in H$ such that $mm_\beta = m'm_\beta$.
Since $M$ is cancellative, one has $m \in H$.

For the case $m \in H$, write $ax_\alpha = \sum_{i = 1}^n r_{i\alpha} x_i$ for some $r_{i\alpha} \in A \cap R$
and $mm_\beta = \sum_{j = 1}^\ell r'_{j\beta} m_j$ for some $r'_{j\beta} \in \kk H$.
Then one has
\begin{eqnarray*}
(a \otimes m)(x_\alpha \otimes m_\beta) & = & ax_\alpha \otimes mm_\beta \\
 & = & \sum_{i = 1}^n \sum_{j = 1}^\ell (r_{i\alpha} \otimes r'_{j\beta}) (x_i \otimes m_j).
\end{eqnarray*}
Therefore the trace of the map given by left multiplication of $a \otimes m$ satisfies
\begin{equation*}
\tr(a \otimes m) = \sum_{\alpha = 1}^n \sum_{\beta = 1}^\ell (r_{\alpha\alpha} \otimes r'_{\beta\beta}) =
\Big(\sum_{\alpha = 1}^n r_{\alpha\alpha}\Big) \otimes \Big(\sum_{\beta = 1}^\ell r'_{\beta\beta}\Big) = \tr(a) \otimes \tr(m). \qedhere
\end{equation*}
\end{proof}

\begin{proof}[Proof of Theorem \ref{thm.twist}]
Using the same notation in the proof of Lemma \ref{lem.trace}, we have that 
$$\{x_i\otimes m_j~|~1 \leq i \leq n~\text{and}~1\leq j \leq \ell\}$$
is a basis of $T$ over $R$.  List this basis in the order
$$\{x_1\otimes m_1,\dots,x_n\otimes m_1,x_1\otimes m_2,\dots,x_n\otimes m_2,\dots,
x_1\otimes m_\ell,\dots,x_n\otimes m_\ell\}.$$
One may then think of the trace matrix corresponding to this ordered basis of $T$ over
$R$ as an $m \times m$ block matrix with blocks of size $n \times n$, where the block
is determined by index on the $\{m_j\}$ basis used.

Since $(x_i\otimes m_j)(x_{i'}\otimes m_{j'}) = x_i(m_jx_{i'})\otimes m_jm_{j'}$, Lemma
\ref{lem.trace} shows that the $(j,j')^\text{th}$ block entry is the
zero matrix whenever $m_jm_{j'} \not\in H$.  By Lemma \ref{lem.monBasis} there
is exactly one $m_{j'}$ such that $m_jm_{j'} \in H$.  In addition, one has
that when $m_jm_{j'} \in R \cap M$, the $(j,j')^\text{th}$ block entry is
$W_{m_j}\tr(m_jm_{j'})$, where $W_{m_j}$ is as in Notation \ref{not.wsigma}.
By Lemma \ref{lem.wsigma} we have that the determinant of $W_{m_j}$ is (up to a unit in $B$) the
determinant of the matrix $W$ of the trace form of $A$ over $B$, which by
definition is the discriminant $d(A/B)$.

Therefore up to a unit in $B$, the determinant of the matrix of the
trace form of $T$ over $R$ is the determinant of the matrix that is
the Kronecker product of the matrices of the trace forms of $A$ over $B$
and of $\kk M$ over $\kk H$.  The claim now follows from the usual
formula for the determinant of a Kronecker product of two matrices.
\end{proof}

\begin{example}
\label{ex.monoid}
Let $A=\kk[x_1,\hdots,x_n]$ and $\sigma \in \Aut(A)$ defined
by $\sigma(x_1) = \xi x_1$, $\xi$ a primitive sixth root of unity,
and $\sigma(x_i)=x_i$ for $i = 2,\hdots,n$.
Then $A^\sigma = \kk[x_1^6,x_2,\hdots,x_n]$ and $d(A/A^\sigma) = x_1^{30}$
by Corollary \ref{cor.refl}.

Let $\rho:\NN \rightarrow \Aut(A)$ be the monoid homomorphism 
sending $1$ to $\sigma$ as in Remark \ref{rem.OreAndSkgpAreTwists}
and let $M$ be the submonoid of $\NN$ generated by $\{2,3\}$.
By restriction, $\rho:M \rightarrow \Aut(A)$ satisfies
$\ker\rho = \{6k \mid k \in \NN\} \subseteq C(M)$
and $\im\rho \cap \Inn(A) = \{\id_A\}$.
Clearly, $\kk M \iso \kk[t^2,t^3]$ and $\kk H \iso \kk[t^6]$.
A basis for $\kk[t^2,t^3]$ as a module over $\kk[t^6]$ is
$\{1,t^2,t^3,t^4,t^5,t^7\}$ and a direct computation shows that
$d(\kk M/\kk H) =_{\kk^\times} t^{42}$.

Consider the twisted tensor product $T = A \tensor_{\tau} \kk[t^2,t^3]$
and let $R = A^\sigma \tensor \kk[t^6]$.
By Theorem \ref{thm.twist}, we have
\[ d(T/R) = d(A/A^\sigma)^6 (d(\kk M/\kk H))^6
	=_{\kk^\times} (x_1^{30} t^{42})^6.\]
\end{example}

\section{Discriminants of Ore extensions} \label{sec.ore}

In this section, we apply Theorem \ref{thm.twist} to the case of an Ore extension.  
Recall that by Remark \ref{rem.OreAndSkgpAreTwists}, Ore extensions are a special case
of the twisted tensor products studied in Sections \ref{sec.twistbg} and \ref{sec.twistThm}.

\begin{theorem}
\label{thm.ore}
Let $A$ be an algebra and set $S=A[t;\sigma]$,
where $\sigma \in \Aut(A)$ has order $m<\infty$ and no
$\sigma^i$, $1 \leq i < m$, is inner.
Suppose $R$ is a central subalgebra of $S$ and set $B = R \cap A^\sigma$.
If $A$ is finitely generated free over $B$ of rank $n$ and $R=B[t^m]$, then
$S$ is finitely generated free over $R$ and
\[ d(S/R) =_{R^\times} \left(d(A/B)\right)^m \left(t^{m-1}\right)^{mn}.\]
\end{theorem}

\begin{proof}
We claim that the hypotheses imply those of Theorem \ref{thm.twist}.
We view $\NN$ as a submonoid of the additive group of the integers
and $\rho:\ZZ \rightarrow \Aut(A)$ as the group homomorphism sending $1$ to $\sigma$.
Then $S \iso A \tensor_\tau \kk \NN$.
Since $\sigma^i$ is not inner for any $i \neq km$, $k \in \ZZ$, then
$\im(\rho\mid_{\NN}) \cap \Inn(A) = \{\id_A\}$ and 
$\ker\rho = \{ km \mid k \in \ZZ\}$.
Set $H = \ker\rho \cap \NN = \{ km \mid k \in \NN\}$,
then $\{1,\hdots,m-1\}$ is a basis for $\kk \NN$ over $\kk H$
implying Theorem \ref{thm.twist} (\ref{thm.twist.c}).
The hypothesis that $A$ is free over $B$ of rank $n$
is equivalent to Theorem \ref{thm.twist} (\ref{thm.twist.a}).
By Lemma \ref{lem.center}, $R = B[t^m] = (A \cap R) \tensor \kk H$.
Hence, Theorem \ref{thm.twist} (\ref{thm.twist.b}) is satisfied.
The formula now follows from Theorem \ref{thm.twist} and Corollary \ref{cor.ztr}.
\end{proof}

\begin{corollary}
Let $A = \kk[x_1,\dots,x_n]$ and $\sigma$ be a reflection of order $m$.  
Let $f$ be a linear form that satisfies $\sigma(f) = \xi f$ 
where $\xi$ is a primitive $m^\text{th}$ root of unity.  
Then the discriminant of the Ore extension $A[t;\sigma]$ is (up to scalar)
$f^{(m-1)m^2}t^{(m-1)m^2}$.
\end{corollary}

\begin{proof}
This follows from applying Corollary \ref{cor.refl} and 
using the Ore extension discriminant formula from Theorem \ref{thm.ore}.
\end{proof}

As test cases, we consider Ore extensions of the ordinary polynomial ring,
the $(-1)$-skew polynomial ring
\[ V_n = \kk_{-1}[x_1,\hdots,x_n],\]
and the $(-1)$-skew Weyl algebra
\[ W_n = \kk\langle x_1,\hdots,x_n \mid x_ix_j+x_jx_i = 1 \text{ for } i \neq j \rangle.\]
Note that $\gr(W_n)=V_n$.

\begin{example}[{\cite[Example 1.7]{CPWZ1}}]
\label{ex.ore1}
$V_2$ is the Ore extension $\kk[x][y;\sigma]$ 
where $\sigma(x)=-x$ and $C(V_2)=\kk[x^2,y^2]$.
Clearly $\kk[x]$ is free over $\kk[x]^\sigma = \kk[x^2]$
and $d(\kk[x]/\kk[x^2]) =_{\kk^\times} x^2$
by Corollary \ref{cor.ztr}.
By Theorem \ref{thm.ore},
\[ d(V_2/C(V_2)) =_{\kk^\times} (x^2)^2 (y)^4 = x^4y^4.\]
\end{example}

\begin{example}
\label{ex.ore2}
By \cite[Lemma 4.1 (3)]{CPWZ1},
\[ C(V_n) = 
\begin{cases}
	\kk[x_1^2,\hdots,x_n^2] & \text{ if $n$ is even} \\
	\kk[x_1^2,\hdots,x_n^2,\prod_i x_i] & \text{ if $n$ is odd.}
\end{cases}\]
Set $C_n=\kk[x_1^2,\hdots,x_n^2]$ regardless of whether $n$ is even or odd.
In either case, $V_n$ is finitely generated free over $C_n$;
this is proved in \cite[Lemma 4.1 (4)]{CPWZ1} for $n$ even
but the proof applies equally well when $n$ is odd.
However, in the case $n$ is odd we do not obtain useful information about
the automorphism group of $V_n$ because a given automorphism may not fix $C_n$.
Regardless, we use Theorem \ref{thm.ore} to inductively compute 
\begin{align}
\label{eq.vncn}
	d(V_n/C_n) =_{\kk^\times} \left( \prod_{i=1}^n x_i^2 \right)^{2^{n-1}}.
\end{align}
This gives an alternate method for obtaining the discriminant in
\cite[Theorem 4.9 (1)]{CPWZ1}.

The case $n=2$ follows from Example \ref{ex.ore1}.
Suppose \eqref{eq.vncn} holds for some $n$ and set
$S=V_n[x_{n+1};\sigma]$ where $\sigma(x_i)=-x_i$ for $i = 1,\hdots,n$.
If $n$ is odd, then $\sigma$ does not fix $\prod_i x_i$.
Hence, $C(V_n) \cap V_n^\sigma = C_n$ in both cases when $n$ is even or odd,
and $V_n$ is finitely generated free over $C_n$ of rank $2^n$.
Thus, by Theorem \ref{thm.ore},
\[ d(S/C_{n+1}) = d(V_n/C_n)^2 (2x_{n+1})^{2 \cdot 2^n}
	=_{\kk^\times} \left( \prod_{i=1}^{n+1} x_i^2 \right)^{2^n}.\]
\end{example}

\begin{example}
\label{ex.ore3}
Let $A=\kk[x,y]$ and $\sigma \in \Aut(A)$ defined by $\sigma(x)=y$ and $\sigma(y)=x$.
Let $S=A[t;\sigma]$.

We have $|\sigma|=2$ and $\sigma$ is not an inner automorphism.
Since $A$ is commutative, $C(A)^\sigma = A^\sigma = \kk[x+y,xy]$.
Thus $C(S)=A^\sigma[t^2]$. 
A basis for $A$ over $A^\sigma$ is $\{1,x\}$.
An easy computation shows that
\[ \tr(1)=2, \;\;\; \tr(x)=x+y, \;\;\; \tr(x^2) = x^2+y^2.\]
Thus, the trace matrix for $A$ over $A^\sigma$ is
\[ \begin{pmatrix}2 & x+y \\ x+y & x^2+y^2\end{pmatrix} \]
and so $d(A/A^\sigma) = (x-y)^2$. 
By Theorem \ref{thm.ore},
\[ d(S/C(S)) =_{\kk^\times} \left( (x-y)^2\right)^2 \left( t^2 \right)^2 = (x-y)^4t^4.\]
The discriminant of $S/C(S)$ is not dominating
in the sense of \cite[Definition 2.1]{CPWZ1}.

The discriminant computation above can also be seen by observing that
$S \iso \kk_{(p_{i,j})}[x_1,x_2,x_3]$ where
$p_{2,3} = p_{3,2} = -1$ and all other $p_{i,j} = 1$.
The isomorphism is given by 
$x_1 \leftrightarrow x+y,\; 
x_2 \leftrightarrow x-y$, and 
$x_3 \leftrightarrow t$.
$S$ is free over its center $C(S)=\kk[x_1,x_2^2,x_3^2]$, 
and the discriminant (up to a constant) is 
$D = x_2^4 x_3^4$ \cite[Proposition 2.8]{CPWZ2}.
\end{example}

\begin{question} 
If we instead take $\sigma \in \Aut(V_2)$ given by
$\sigma(x)=-y$, $\sigma(y)=x$
and set $S=A[t;\sigma]$ so that $S$ satisfies
\[ xy = -yx,\;\; tx = y t,\;\; ty= -xt,\]
what is the discriminant $d(S/C(S))$?

Changing to generators that include the eigenvectors of $\sigma$ 
does not give a skew-polynomial ring (as it did in the previous example). 
Because $\sigma^2$ is inner, 
Theorem \ref{thm.ore} does not apply.
In particular, $C(S)=\kk[x^2 + y^2, x^2y^2, xyt^2, t^4]$ is not a UFD. 
\end{question}

We are interested in the Ore extension 
$W_2[t;\sigma]$ with $\sigma(x)=y$ and $\sigma(y)=x$.
Because $\gr(W_2)=V_2$, the discriminant $d(V_2/C(V_2)^\sigma)$ 
is a filtered version of the discriminant of $d(W_2/C(W_2)^\sigma)$.
The Macaulay2 routines are not currently equipped to handle the
computations for this discriminant.
Instead, we pass to the homogenization of $W_2$.

For each $1 \leq i \leq n$, fix $\deg(x_i) \in \ZZ^+$.  This defines a $\NN$-grading
on $\kk\langle x_1,\hdots,x_n\rangle$. If $f = \sum_{i=0}^d f_k$ where $f_k$ is the homogeneous
component of $f$ with $\deg(f_k)=k$ and $f_d \neq 0$,
then the {\sf homogenization of $f$} by the central indeterminate 
$t$ is then $H(f) = \sum_{i=0}^d f_k t^{d-k}$ where $d=\deg(f)$.
It is clear that $H(f)$ is homogeneous in the ring $\kk\langle x_1,\hdots,x_n\rangle[t]$
where $t$ has been assigned degree 1.

Suppose $A$ is an algebra generated by $\{x_1,\hdots,x_n\}$
subject to the relations $r_1,\hdots,r_m$
and such that $\deg(x_i) > 0$.
The {\sf homogenization} $H(A)$ of $A$ is the 
algebra on the generators $\{t,x_1,\hdots,x_n\}$
subject to the homogenized relations $H(r_i)$, $i=1,\hdots,m$, 
as well as the additional relations $tx_j-x_jt$, 
$1 \leq j \leq n$.

\begin{theorem}
\label{thm.homog}
Suppose $A$ is an algebra generated by $\{x_1,\hdots,x_n\}$
subject to the relations $r_1,\hdots,r_m$
and such that $\deg(x_i) > 0$.
If $A$ is finitely generated free over a central subalgebra $R$,
then $H(A)$ is finitely generated free over $H(R)$ and 
\[ d(H(A)/H(R)) =_{(H(R))^\times} H(d(A/R)).\]
\end{theorem}

\begin{proof}
Suppose $A$ (and hence $H(A)$) is generated in degree 1.
This is easily generalized to other cases.
There is an isomorphism $H(A)[t\inv] \rightarrow A[t^{\pm 1}]$
fixing $t$ and for $i=1,\hdots,n$, $x_i \mapsto t\inv x_i$.
By \cite[Lemma 1.3]{CYZ1} and \cite[Lemma 3.1]{CPWZ1}
\[ d(H(A)[t^{\pm 1}]/R[t^{\pm 1}]) 
=_{(R[t^{\pm 1}])^\times} d(A[t]/R[t]) 
=_{(R[t])^\times} d(A/R).\]
Tracing back through the isomorphism and clearing fractions gives the result.
\end{proof}

\begin{example}
\label{ex.homog1}
Let $A=W_2$, the 2-dimensional $(-1)$-quantum Weyl algebra
$A = \kk\langle x,y \mid xy+yx=1\rangle$.
Note that $C(A)=\kk[x^2,y^2]$.
By \cite[Theorem 0.1]{CYZ1},
$d(A/C(A)) =_{\kk^\times} \left( 4 x^2y^2-1 \right)^2$.

It follows from \cite[Proposition 2.8]{gadpbw}
that $C(H(A))=\kk[x^2,y^2,t]$.
Hence, by Theorem \ref{thm.homog}
\[ d(H(A)/C(H(A))) =_{\kk^\times} \left( 4 x^2y^2-t^4 \right)^2.\]
\end{example}

\begin{example}
\label{ex.homog2}
Let $A$ be as in the previous example and
let $\sigma$ be the automorphism $x \leftrightarrow y$.
Then $\gr(A) = V_2$ and
$C(A) = C(V_2) = \kk[x^2,y^2]$.
Moreover,
$C(A)^\sigma = C(V_2)^\sigma = \kk[x^2+y^2,x^2y^2]$.
Extend $\sigma$ to $H=H(A)$ by $\sigma(t)=t$.
Then $C(H)^\sigma = \kk[x^2+y^2,x^2y^2,t]$
so $\rank(A/C(A)^\sigma) = \rank(H/C(H)^\sigma)=8$.
Let $X=x^2+y^2$, $Y=x^2y^2$, and $T=t$. Then
\[ d(H/C(H)^\sigma) =_{\kk^\times} (4Y-T^4)^4 (X^2-4Y)^4.\]
By \cite[Proposition 4.7]{CPWZ2},
\begin{align*}
	d(A/C(A)^\sigma) &=_{\kk^\times} (4Y-1)^4 (X^2-4Y)^4, \text{ and } \\
	d(V_2/C(V_2)^\sigma) &=_{\kk^\times} Y^4 (X^2-4Y)^4.
\end{align*}
\end{example}

\section{Skew group algebras}
\label{sec.skgrp}

\begin{theorem}
\label{thm.skgrp}
Let $A$ be an algebra and $G$ a finite group that acts on $A$ as automorphisms
such that no non-identity element of $G$ is inner.
Set $S=A\# G$ and identify $A$ with its image under
the embedding $a \mapsto a \tensor e$.
Suppose $A$ is a finitely generated free over 
the subalgebra $R \subseteq C(A)^G$. 
Then $S$ is finitely generated free over $R$ and
\[ d(S/R) =_{R^\times} d( A/R )^{|G|}.\]
\end{theorem}

\begin{proof}
This follows almost immediately from Theorem \ref{thm.twist}.
By hypothesis, there is a map $\rho:G \rightarrow \Aut(A)$,
$\im\rho \cap \Inn(A) = \{\id_A\}$, and $H=\ker\rho = \{e_G\}$.
Our hypotheses directly imply 
(\ref{thm.twist.a}) and (\ref{thm.twist.b}) in Theorem \ref{thm.twist}.
Because the elements of $G$ form a basis of $\kk[G]$,
we have $\ell = |G|$.
\end{proof}

\begin{example}
\label{ex.skgrp2}
Let $A=\kk[x_1,x_2,x_3]$ and $G=\cS_3$,
the symmetric group acting as permutations of $x_i$.
By Example \ref{ex.discS3},
\[ d(A/A^G) = \left[\prod_{i < j} (x_i-x_j) \right]^6.\]
Set $S=A\# G$ and $R=A^G$ identified both in $A$ and in $C(S)$. 
It follows from Theorem \ref{thm.skgrp} that
\[ d(S/R) =_{R^\times} \left[\prod_{i < j} (x_i-x_j) \right]^{36} \tensor e.\]
\end{example}

We are interested in the 
skew group algebra $V_n \# \cS_n$ 
where $\cS_n$ is the symmetric
group on $n$ letters
acting as permutations on the $x_i$.
We have that $C(V_n \# \cS_n)$ may be 
identified with $C(V_n)^{\cS_n}$.
In the case when $n$ is even we can
describe this center explicitly.

\begin{lemma} \label{lem.noInnerSn}
Let $\cS_n$ act on $V_n$ as permutations of the variables and let
$\Inn(V_n)$ denote the set of inner automorphisms induced
by normal elements of $V_n$.   Then $\cS_n \cap \Inn(V_n) = \{e\}$.
\end{lemma}

\begin{proof}
Let $\sigma$ be a nontrivial permutation of $\{1,\dots,n\}$,
and suppose that $\sigma$ is an inner automorphism induced by
the normal element $a \in V_n$.  Choose $i$ such that $\sigma(i) \neq i$.
Then if one considers the equality $ax_i = x_{\sigma(i)}a$,
one sees this is impossible since the set of monomials that appear on the left hand
side is disjoint from the set of monomials which appear on the right hand side.
\end{proof}

\begin{lemma}
Let $E_n=\kk[e_1,\hdots,e_n]$ where the $e_i$
are the elementary symmetric functions
in the $x_1^2,\hdots,x_n^2$.
If $n$ is even, then $C(V_n)^{S_n} = E_n$
and $V_n$ is free over $E_n$ of order $2^n n!$.
Consequently, $V_n \# \cS_n$ is finitely generated
free over its center of order $2^n (n!)^2$.
\end{lemma}

\begin{proof}
The elementary symmetric functions satisfy $\deg e_i = 2i$. 
Set $E_n=\kk[e_1,\hdots,e_n]$.
We claim \[\rank(V_n/E_n)=2^n n!.\]
The Hilbert series of $V_n$ is $H_{V_n}(t)=1/(1-t)^n$ while that for $E_n$ is
\[ H_{E_n}(t) = \frac{1}{(1-t^2)(1-t^4)\cdots(1-t^{2n})}.\]
Let $H_n(t) = H_{V_n}(t)/H_{E_n}(t)$ and assume inductively that $H_n(1)=2^n n!$.
This clearly holds in the case $n=1$. Thus,
\[ 
H_{n+1}(t) 
	= \frac{(1-t^2)(1-t^4)\cdots(1-t^{2(n+1)})}{(1-t)^{n+1}}
	= H_{n}(t) \cdot \frac{(1-t^{2(n+1)})}{1-t} 
	= H_n \cdot (1+t^{n+1})(1 + t + t^2 + \cdots + t^n).
\]
Hence, $H_{n+1}(1) = H_n(1) \cdot 2 \cdot (n+1) = 2^{n+1} (n+1)!$.
Since $V_n \# \cS_n$ has rank $n!$ over $V_n$
it follows that it has rank $2^n (n!)^2$ over $E_n$.

Freeness follows from the Auslander-Buchsbaum formula.
Since $E_n$ is a polynomial ring then 
$\pd_{E_n}(V_n) = \depth_{E_n}(V_n) - \depth(E_n) = 0$.

That the center of $V_n \# \cS_n$
is generated by the elementary symmetric functions follows from
\cite[Lemma 4.1 (3)]{CPWZ1} and Lemma \ref{lem.center} as no element of $\cS_n$
acts as an inner automorphism by Lemma \ref{lem.noInnerSn}.
\end{proof}

When $n$ is odd the center of $V_n$ is not a polynomial ring
and it follows that $C(V_n \# \cS_n)$ is also not a polynomial ring.

\begin{example}
\label{ex.skgrp1}
Let $\cS_2$ act on $V_2$ as above
and set $S = V_2 \# \cS_2$.
Then $E_2 = C(S) = \kk[X,Y]$ 
where $X=x^2+y^2$ and $Y=x^2y^2$.
Since $d(V_2/E_2) = Y^2 (X^2-4Y)^2$,
then by Theorem \ref{thm.skgrp},
$d(S/C(S))=_{\kk^\times} [Y^2 (X^2-4Y)^2]^2 \tensor e$.
\end{example}

\section{Automorphism groups}
\label{sec.autos}

In this section we apply our results on the discriminant
to compute explicitly the automorphism groups
in several cases.

\subsection{An Ore extension of $\kk[x,y]$}

Let $A=\kk[x,y]$ and $\sigma \in \Aut(A)$ 
defined by $\sigma(x)=y$ and $\sigma(y)=x$.
Let $S=A[t;\sigma]$, so that $S$ satisfies the relations
\[ xy=yx, \;\;\; tx=yt, \;\;\; ty=xt.\]

By Example \ref{ex.ore3},
$f:=d(S/C(S))=16(x-y)^4t^4$.
Set $X=x+y$, $Y=xy$, and $T=t^2$, so that
$f=16(X^2-4Y)^2 T^2$.
Any automorphism of $S$ preserves the center 
and hence the discriminant up to scalar.
Because the center is a UFD, we have that any automorphism either
preserves the factors $(X^2-4Y)$ and $T$, 
or else it interchanges them (up to a scalar).

\begin{proposition}
With notation above, $\Aut(S)$ consists of maps of the following form ($a,b,c \in \kk^\times$, $d \in \kk$).
\begin{align*}
\begin{pmatrix} x \\ y \\ t \end{pmatrix} &\mapsto
\begin{pmatrix} a & b & 0 \\ b & a & 0 \\ 0 & 0 & c\end{pmatrix}
\begin{pmatrix} x \\ y \\ t \end{pmatrix} + \begin{pmatrix} d \\ d \\ 0 \end{pmatrix}, \\
\begin{pmatrix} x \\ y \\ t \end{pmatrix} &\mapsto
\begin{pmatrix} a & a & -b \\ a & a & b \\ -c & c & 0\end{pmatrix}
\begin{pmatrix} x \\ y \\ t \end{pmatrix} + \begin{pmatrix} d \\ d \\ 0 \end{pmatrix}.
\end{align*}
\end{proposition}

\begin{proof}
Let $g \in \Aut(S)$ and suppose $g$ preserves the factors (up to scalar multiple).
Then $\deg(g(X^2)) \leq 2$ so $\deg(g(X))=1$.
Similarly, $\deg(g(Y)) = 2$ so $\deg(g(x))=\deg(g(y))=1$.
Moreover, $\deg(g(T))=2$ so $\deg(g(t))=1$ and $t$ is mapped to a scalar multiple of itself.
Thus, we reduce to a linear algebra problem and conclude that
all such $g$ have the first form above.

A similar argument follows in the case that $g$ interchanges the factors.
These automorphisms have the second form above.
\end{proof}

All automorphisms of $A$ are triangular, in the sense of \cite[Theorem 3(2)]{CPWZ1}.  
The automorphisms of $A$ are (-1)-affine \cite[Definition 1.7]{CPWZ2}, but not affine.  
$A$ is a skew-polynomial ring that satisfies H2, 
but not H1 of \cite[p.12]{CPWZ2}.  
We conjecture that $\Aut(A)$ is not tame,
note that \cite[Proposition 4.5]{CPWZ2} does not apply 
because $g(X)$ can contain a constant. 
(See the definitions of {\em elementary} and {\em tame} on p. 3 of \cite{CPWZ2}.)

\subsection{An Ore extension of $V_2$}
\label{ssec.v2ore}

Let $A=V_2$ with $\sigma \in \Aut(A)$ given by $\sigma(x)=y$, $\sigma(y)=x$.
Set $S = A[t;\sigma]$ so that $S$ satisfies
\[ xy = -yx,\;\; tx = y t,\;\; ty= xt.\]
This example cannot be reduced to the skew polynomial case by using eigenvectors.
Here $A$ is not free over 
$A^\sigma = \kk\langle x+y, x^3+y^3 \rangle$ 
and $A^\sigma$ is not AS regular.
However, $A$ is free over the polynomial ring $C(A)^\sigma = \kk[x^2+y^2,x^2y^2]$, 
and $C(S) = \kk[x^2+y^2,x^2y^2, t^2]= \kk[X,Y,T]$ is again a polynomial ring.  
By Example \ref{ex.homog2} and Theorem \ref{thm.ore},
\[ d(S/C(S)) =_{\kk^\times} T^8Y^8(X^2-4Y)^8.\]

\begin{proposition}
\label{prop.v2ore}
With notation above, $\Aut(S) \iso \kk^2 \rtimes \{\tau\}$ where $\tau$ is
the automorphism $x \leftrightarrow y$.
\end{proposition}

\begin{proof}
We will apply the discriminant to show that all automorphisms are affine.
The relations of $S$ then imply that all automorphisms are in fact graded.
Once shown, it follows easily that if $g \in \Aut(S)$, then there exists
$a,b \in \kk^\times$ such that
\[
g(x) = ax, g(y) = ay, g(t) = bt \quad\text{or}\quad
g(x) = a y, g(y) = ax, g(t) = bt.\]

Let $g \in \Aut(S)$.
The discriminant is not dominating, but the center is a UFD and 
hence there are six cases for how $g$ permutes the factors of the discriminant.

{\parindent0pt
{\bf Case 1: $g(X^2-4Y) = \alpha(X^2 - 4Y), g(Y) = \beta Y, g(T) = \gamma T$.}

Then 
$g(X)^2 
	= g(X^2) = \alpha(X^2 - 4Y) + 4 g(Y) 
	= \alpha(X^2 - 4Y) + 4 \beta Y 
	= \alpha X^2 + 4(\beta-\alpha) Y$.  
For $\alpha X^2 + 4(\beta-\alpha) Y$ to be the square 
of some polynomial in $\kk[X,Y,T]$ we need $\alpha = \beta$, and then 
$g(X) = \sqrt{\alpha} X$, $g(Y) = \alpha Y$.
Then $g(Y) = g(x^2y^2) = -g((xy)^2) = -\alpha (xy)^2$ so that $g(xy)$ has degree 2
and so $g(x)$ and $g(y)$ have degree 1.  
If $g(T) = \gamma T$, then $g(t) = \sqrt{\gamma} t$. 
Hence $g$ is affine.

The other cases are easily reduced to case 1 or eliminated.

{\bf Case 2: $g(X^2-4Y) = \alpha(X^2 - 4Y), g(Y) = \beta T, g(T) = \gamma Y$.}

$g(X)^2= g(X^2) = \alpha(X^2-4Y) + 4 \beta T$ cannot happen in $\kk[X,Y,T]$.

{\bf Case 3: $g(X^2-4Y) = \alpha T, g(Y) = \beta Y, g(T) = \gamma (X^2-4Y)$.}
$g(X)^2= g(X^2) = \alpha T + 4 \beta Y$ cannot happen in $\kk[X,Y,T]$.

{\bf Case 4: $g(X^2-4Y) = \alpha Y, g(Y) = \beta (X^2 - 4Y), g(T) = \gamma T$.}
$g(X)^2= g(X^2) = \alpha Y + 4 \beta (X^2-4Y)$ so $\alpha = 16 \beta$ and 
$g(X) = 2 \sqrt{\beta} X$ and $g(Y) = -g(xy)^2 = \beta(X^2-4Y)$
so $g(x)$ and $g(y)$ have degree 1; 
further $g(t) = \sqrt{\gamma} t$ $g$ is affine.

{\bf Case 5: $g(X^2-4Y) = \alpha T, g(Y) = \beta (X^2 - 4Y), g(T) = \gamma Y$.}
$g(X)^2= g(X^2) = \alpha T + 4 \beta (X^2-4Y)$ cannot happen in $\kk[X,Y,T]$.

{\bf Case 6: $g(X^2-4Y) = \alpha Y, g(Y) = \beta T, g(T) = \gamma (X^2 - 4Y)$.}
$g(X)^2= g(X^2) = \alpha Y + 4 \beta T$ cannot happen in $\kk[X,Y,T]$.
}
\end{proof}

\subsection{An Ore extension of $V_3$}

Let $R=\kk_{-1}[x,y,z]$ and $\sigma$ the automorphism that interchanges $x$ and $y$.
Then $C(R[t;\sigma]) = \kk[x^2 + y^2, x^2y^2, z^2, t^2]$
($\sigma$ eliminates $xyz$ from $C(R)$).  
$\Aut(R)$ contains a free subgroup on two generators.

\begin{question} 
Does $\Aut(R[t;\sigma])$ also contain a free subgroup on two generators?
\end{question}

\subsection{An Ore extension of $W_2$}

Let $\sigma \in \Aut(W_2)$ be
given by $\sigma(x)=y$ and $\sigma(y)=x$.
Set $S=W_2[t;\sigma]$ so that $S$ satisfies
\[ xy+yx=1, \;\; tx=yt, \;\; ty = xt.\]
The center of $S$ is $C(S) = \kk[x^2+y^2,x^2y^2, t^2]$.
Set $X=x^2+y^2$, $Y=x^2y^2$, and $T=t^2$.
By Example \ref{ex.homog2} and Theorem \ref{thm.ore},
$d(S/C(S)) = T^8(4Y-1)^8(X^2 - 4Y)^8$.

\begin{proposition}
With notation above, $\Aut(S) \iso (\kk^\times \times \{-1,1\}) \rtimes \{\tau\}$ where $\tau$ is
the automorphism $x \leftrightarrow y$.
\end{proposition}

\begin{proof}
The proof here is nearly identical to that in Proposition \ref{prop.v2ore}.
We will apply the discriminant to show that all automorphisms are affine.
The relations of $S$ then imply that all automorphisms are in fact graded.
Once shown, it follows easily that if $g \in \Aut(S)$, then there exists
$a \in \{-1,1\}$ and $b \in \kk^\times$ such that
\[
g(x) = a y, g(y) = a\inv x, g(t) = bt \quad\text{or}\quad
g(x) = ax, g(y) = a\inv y, g(t) = bt.\]

Let $g \in \Aut(S)$.
The discriminant is not dominating, but the center is a UFD and 
hence there are six cases for how $g$ permutes the factors of the discriminant.

{\parindent0pt
{\bf Case 1: $g(X^2-4Y) = \alpha(X^2 - 4Y), g(4Y-1) = \beta (4Y-1), g(T) = \gamma T$.}
Then 
$g(X)^2 
	= g(X^2) = \alpha(X^2 - 4Y) + g(4Y) 
	= \alpha(X^2 - 4Y) + \beta(4Y-1) + 1
	= \alpha X^2 + 4(\beta-\alpha) Y + (1-\beta)$.  
For $\alpha X^2 + 4(\beta-\alpha) Y$ to be the square 
of some polynomial in $\kk[X,Y,T]$ we need $\alpha = \beta=1$, and then 
$g(X) = X, g(Y) = Y$.
Then $g(Y) = g(x^2y^2) = -g((xy)^2) = -(xy)^2$ 
so that $g(xy)$ has degree 2
and so $g(x)$ and $g(y)$ have degree 1.
If $g(T) = \gamma T$, then $g(t) = \sqrt{\gamma} t$. 

The other cases are easily reduced to case 1 or eliminated.

{\bf Case 2: $g(X^2-4Y) = \alpha(X^2 - 4Y), g(4Y-1) = \beta T, g(T) = \gamma (4Y-1)$.}
$4Y-1$ has degree $4$ and it follows that $g(4Y-1)$ has degree $4$.
Since $T$ is of degree $2$, we cannot have $g(4Y-1)=\beta T$.

{\bf Case 3: $g(X^2-4Y) = \alpha T, g(4Y-1) = \beta (4Y-1), g(T) = \gamma (X^2-4Y)$.}
$g(X)^2= g(X^2) = \alpha T + 4 \beta (4Y-1)$ is not a square in $\kk[X,Y,T]$.

{\bf Case 4: $g(X^2-4Y) = \alpha (4Y-1), g(4Y-1) = \beta (X^2 - 4Y), g(T) = \gamma T$.}
$g(X)^2= g(X^2) = \alpha (4Y-1) + 4 \beta (X^2-4Y) = 2\beta X^2 + 4(\alpha-4\beta)Y - \alpha$.
This is not a square in $\kk[X,Y,T]$.

{\bf Case 5: $g(X^2-4Y) = \alpha T, g(4Y-1) = \beta (X^2 - 4Y), g(T) = \gamma (4Y-1)$.}
$g(X)^2= g(X^2) = 4 \alpha T + 4 \beta (X^2-4Y)$ is not a square in in $\kk[X,Y,T]$.

{\bf Case 6: $g(X^2-4Y) = \alpha (4Y-1), g(4Y-1) = \beta T, g(T) = \gamma (X^2 - 4Y)$.}
See Case 2.
}
\end{proof}

\subsection{The homogenization of $W_2$.}

Let $H=H(W_2)$ and $C=C(H)$.
In Example \ref{ex.homog1} it was shown that
$d(H/C) =_{\kk^\times} \left( 4 x^2y^2-t^4 \right)^2$.

$H$ is $\NN$-graded, so for $h \in H$, 
denote by $h_d$ the degree $d$ component

\begin{proposition}
With notation above, $\Aut(H) \iso (\kk^\times)^2 \rtimes \{\tau\}$
where $\tau$ is the automorphism $x \leftrightarrow y$.
\end{proposition}

\begin{proof}
Let $I$ be a height one prime ideal of $H$.
By \cite[Theorem 6.6]{gaddis},
either $I=(t)$, $I=(xy-yx)$, or
$I=(g)$ with $\deg(g)>1$.
Given $\phi \in \Aut(H)$, it follows that 
$\phi(\deg(r)) \geq \deg(r)$ and so
$(t)$ is the only height one prime ideal
generated by a degree one element.
Hence $\phi(t) = \alpha t$ for some $\alpha \in \kk^\times$.
Thus, $\deg(\phi(t^2)) = 2$ and so $\deg(\phi(x^2y^2)) = 4$.
We conclude that $\phi$ is affine.

Let $\phi \in \Aut(H)$ and write
\[ 
\phi(x) = a_0 + a_1 x + a_2 y + a_3 t, \;\;\; 
\phi(y) = b_0 + b_1 x + b_2 y + b_3 t, \;\;\; 
\phi(t) = c_0 + c_1 x + c_2 y + c_3 t,\]
with $a_i,b_i,c_i \in \kk$ for $i=0,\hdots,3$. 
Because $t$ is central, then $c_1=c_2=0$.
Hence,
\[ 0 = \phi(xy+yx-t^2)_0 = 2a_0b_0 - c_0^2\]
and
\begin{align*}
0 	&= \phi(xy+yx-t^2)_1  \\
	&= 2 \left[ a_0 (b_1 x + b_2 y + b_3 t) + b_0 (a_1 x + a_2 y + a_3 t) - c_0c_3t\right] \\
	&= 2 \left[ (a_0b_1 + b_0a_1)x + (a_0b_2+b_0a_2)y + (a_0b_3+b_0a_3 - c_0c_3)t\right].
\end{align*}
If $b_0=0$, then $c_0=0$ and
$a_0b_1=a_0b_2=a_0b_3=0$.
Since $\phi_1(y) \neq 0$, then $a_0=0$.
Suppose $b_0 \neq 0$, then 
$-\frac{a_0}{b_0} = \frac{a_1}{b_1} = \frac{a_2}{b_2}$
so $a_1b_2-a_2b_1=0$ and $\phi$ is not an isomorphism.
Hence, we conclude that $a_0=b_0=c_0=0$.
Thus,
\begin{align*}
0 	&= \phi(xy+yx-t^2) \\
	&= 2(a_1b_1 x^2 + a_2b_2 y^2) + (a_1b_2+a_2b_1)xy + (a_2b_1+a_1b_2)yx - c_3^2 t^2 \\
	&= 2(a_1b_1 x^2 + a_2b_2 y^2) + (a_1b_2 + a_2b_1-c_3^2)t^2.
\end{align*}
We have two cases. Either $a_1=b_2=0$ or $a_2=b_1=0$
and $c_3$ is determined by the $a_i,b_j$.
\end{proof}

\subsection{The automorphism group of $V_2 \# S_2$}

Set $A=V_2 \# S_2$ and write $\cS_2 = \{e,g\}$ as before.
Example \ref{ex.skgrp1} shows that 
$d(A/C(A))=_{\kk^\times} [Y^2 (X^2-4Y)^2]^2 \tensor e$
where $X=x^2+y^2$ and $Y=x^2y^2$.

Because $C(A) \iso E_2$ is a PID,
then any automorphism of $A$
either preserves the factors $Y$ and $X^2-4Y$ 
or else it interchanges them
(up to a scalar).
Suppose $\phi \in \Aut(V_2)$.
It follows easily that $\deg(\phi(Y)) \leq 4$.
If $\phi$ preserves the factors $Y$ and $X^2-4Y$,
then $\phi(Y)=k_1Y$ and $\phi(X^2-4Y)=k_2(X^2-4Y)$
for $k_1,k_2 \in \kk^\times$. We have
\[ k_2(X^2-4Y) = \phi(X^2-4Y) = \phi(X)^2 - 4k_1Y.\]
Thus, $\phi(X)^2 = k_2(X^2-4Y) - 4k_1Y$.
As $V_2$ is a domain and 
the degree of the right-hand side is at most 4, 
then the degree of $\phi(X)$ is at most 2.
A similar argument shows the same result
when $\phi$ interchanges the factors.

\begin{lemma}
\label{lem.idimg}
Let $\phi \in \Aut(A)$,
then $\phi(1 \tensor g) = \pm(1 \tensor g)$.
\end{lemma}

\begin{proof}
Write $\phi(1 \tensor g) = a \tensor e + b \tensor g$.
We have $(1 \tensor g)^2 = 1 \tensor e$, so
\[ 1 \tensor e 
	= (a \tensor e + b \tensor g)^2
	= (a^2 + b (g.b)) \tensor e + (ab + b (g.a)) \tensor g.
\]
Hence, $a^2 + b (g.b) = 1$ and $ab + b (g.a) = 0$.
Write $a = a_0 + a_1 + \cdots + a_d$ where $\deg(a_k)=k$
and similarly for $b$.
We have $0 = (ab + b (g.a))_0 = 2a_0b_0$
and $1 = (a^2 + b (g.b))_0 = a_0^2 + b_0^2$.
Thus, either $a_0 = \pm 1$ and $b_0 = 0$, or $b_0 = \pm 1$ and $a_0 = 0$.

Suppose $a_0=1$ and $b_0 = 0$. The remaining cases are similar.
Then $0 = (a^2 + b (g.b))_1 = 2a_1$, so $a_1=0$,
and $0=(ab + b (g.a))_1 = 2a_0b_1$, so $b_1=0$.

We proceed by induction. 
Suppose $a_k = b_k = 0$ for all $k=1,\hdots,n-1$. Then
\[ 0 = (a^2 + b (g.b))_n = \left((a_0 + a_n)^2\right)_n = 2a_0a_n,\]
so $a_n = 0$. Furthermore,
$0 = (ab + b (g.a))_n = 2a_0b_n$, so $b_n=0$.
\end{proof}

Throughout, let $\phi \in \Aut(A)$ and write
$\phi(x \tensor e) = r \tensor e + s \tensor g$.

As a consequence of the Lemma \ref{lem.idimg} we have
\[\phi(y \tensor e)
= \phi((1 \tensor g)(x \tensor e)(1 \tensor g))
= (1 \tensor g)\phi(x \tensor e)(1 \tensor g).\]
Hence,
\[\phi(y \tensor e) 
= (1 \tensor g)(r \tensor e + s \tensor g)(1 \tensor g)
= g.r \tensor e + g.s \tensor g.\]
Moreover, $x \tensor g = (x \tensor e)(1 \tensor g)$
and $y \tensor g = (y \tensor e)(1 \tensor g)$.
Thus, the automorphism $\phi$ is completely determined
by the choice of $r$ and $s$.

Hence, we have the equations,
\begin{align}
\label{eq1}
\phi((x^2+y^2) \tensor e)
	&= \left(r^2 + s(g.s) + g.r^2 + (g.s)s\right) \tensor e
		+ \left(rs+s(g.r) + g.(rs) + (g.s)r\right) \tensor g, \\
\label{eq2}
\phi((xy+yx) \tensor e)
	&= \left(r(g.r) + s^2 + (g.r)r + (g.s)^2\right) \tensor e
		+ \left(r(g.s) + sr + (g.r)s + (g.s)(g.r)\right) \tensor g.
\end{align}

\begin{lemma}
The degree zero components of $r$ and $s$ are zero.
\end{lemma}

\begin{proof}
Since $xy+yx=0$, one has $0 = \phi((xy+yx) \tensor e)$.
By restricting \eqref{eq2} to the degree
zero component we find $r_0^2+s_0^2=0$ and $r_0s_0=0$.
The result now follows.
\end{proof}

\begin{lemma}
\label{lem.eqdeg}
Suppose $\deg(r) > \deg(s) \geq 1$, then $s=0$.
Similarly, if $\deg(s)>\deg(r) \geq 1$, then $r=0$.
\end{lemma}

\begin{proof}
Suppose $\deg(r) > \deg(s) \geq 1$ and
write $r=r_1 + \cdots + r_d$ where $\deg(r_k)=k$
and by hypothesis $d>1$.
Because $\phi$ is an automorphism,
then $\deg(\phi((x^2+y^2) \tensor e)) \leq 2$,
as was noted before Lemma \ref{lem.idimg}.
By \eqref{eq1},
$\left(r^2 + s(g.s) + g.r^2 + (g.s)s\right)_{2d}=0$,
then we have $r_d^2 + (g.r_d)^2 = 0$.

Because the action of $g$ is diagonalizable,
we can decompose $r_d$ uniquely
into a sum of elements from the two weight spaces,
so $r_d=r_+ + r_-$ where $g.r_+ = r_+$ and $g.r_- = -r_-$.
We then have
\[ 0 = r_d^2 + (g.r_d)^2 = 2(r_+^2 + r_-^2).\]
Because the weight spaces are disjoint, we conclude that $r_d=0$.

A similar argument holds in the case
$\deg(s)>\deg(r)\geq 1$ 
but we use \eqref{eq2} instead of \eqref{eq1}.
\end{proof}

Write $\hat{r}_k = r_k + g.r_k$ and $\hat{s}_k = s_k + g.s_k$
so that both $\hat{r}_k$ and $\hat{s}_k$
are fixed by the action of $g$
and $\hat{r}_k=0$ if and only if 
$r_k$ belongs to the negative weight space.
Since $(x+y)^2 = x^2 + y^2$ in $V_2$, then
\begin{align*}
\phi( (x^2+y^2) \tensor e )
	&= \phi\left( (x+y) \tensor e \right)^2 
	= \left[\sum_{k=1}^d \hat{r}_k \tensor e 
			+ \hat{s}_k \tensor g\right]^2.
\end{align*}
Let $\ell \in \{2,\hdots,d\}$ be the largest degree
such that the above expression is nonzero.
Then we have
\begin{align*}
0 	&= (\hat{r}_\ell \tensor e + \hat{s}_\ell \tensor g)^2 \\
	&= \left(\hat{r}_\ell^2 
		+ \hat{s}_\ell(g.\hat{s}_\ell) \right) \tensor e 
	+ \left(\hat{r}_\ell\hat{s}_\ell 
		+ \hat{s}_\ell(g.\hat{r}_\ell)\right) \tensor g  \\
	&= \left(\hat{r}_\ell^2 
		+ \hat{s}_\ell^2 \right) \tensor e 
	+ \left(\hat{r}_\ell\hat{s}_\ell 
		+ \hat{s}_\ell\hat{r}_\ell\right) \tensor g.
\end{align*}
Each component must be zero and so
$(\hat{r}_\ell+ \hat{s}_\ell)^2=0$. 
Thus, $\hat{r}_\ell=-\hat{s}_\ell$ but because 
$\hat{r}_\ell^2 + \hat{s}_\ell^2=0$ then 
$\hat{r}_\ell=\hat{s}_\ell=0$.
Hence, all higher degree components of $r$ and $s$ are contained
in the negative weight space.

Write 
$\phi( (x+y) \tensor e) = u \tensor e + v \tensor g$
with $u,v \in V_2$.
It follows from Lemma \ref{lem.eqdeg} that $d=\deg(u)=\deg(v)$.
Moreover, $u_k,v_k$ are contained in the negative weight space for $k>1$. 
Then we have
\[
\phi( (x^2+y^2) \tensor e) 
	= \phi \left( [(x+y) \tensor e ]^2 \right)
	= \left[ u \tensor e + v \tensor g \right]^2
	= (u^2-v^2) \tensor e + (uv - vu) \tensor g.
\]
Assume $d>1$.
In the top degree we have
$(u_d^2-v_d^2)=0$ and $(u_dv_d-v_du_d)=0$
so $(u_d-v_d)(u_d+v_d)=0$.
Hence, $u_d = \pm v_d$.

\noindent {\bf Case 1 ($u_d=v_d$):}
We claim $u_k=v_k$ for all $k \leq d$.
Suppose this holds for some $\ell \leq d$.
\begin{align*}
0 
&= \left[ u^2 - v^2 \right]_{d+\ell-1}
 = \left[ (u_1 + \cdots + u_d)^2 
	- (v_1 + \cdots + v_d)^2 \right]_{d+\ell-1} \\
&= \left[ (u_1 + \cdots + u_d)^2
	-(v_1 + \cdots + v_{\ell-1} + u_{\ell} + \cdots + u_d)^2 \right]_{d+\ell-1} \\
&= u_{\ell-1}u_d + u_du_{\ell-1} - v_{\ell-1}u_d - u_dv_{\ell-1}. \\
0 
&= [uv-vu]_{d-\ell+1}
 = \left[ (u_1 + \cdots u_d)(v_1 + \cdots + v_d)
 		- (v_1 + \cdots + v_d)(u_1 + \cdots u_d) \right]_{d-\ell+1} \\
& = \left[ 
(u_1 + \cdots u_d)(v_1 + \cdots + v_{\ell-1} + u_{\ell} + \cdots + u_d)
- (v_1 + \cdots + v_{\ell-1} + u_{\ell} + \cdots + u_d)(u_1 + \cdots u_d) \right]_{d-\ell+1} \\
&= u_{\ell-1}u_d + u_dv_{\ell-1} - v_{\ell-1}u_d - u_du_{\ell-1}.
\end{align*}
Combining these gives
\begin{align*}
	0 = u_du_{\ell-1} - u_dv_{\ell-1} = u_d(u_{\ell-1}-v_{\ell-1}).
\end{align*}
Hence, $u_{\ell-1}=v_{\ell-1}$.

\noindent {\bf Case 2 ($u_d=-v_d$):}
This case follows similarly to the above.

We conclude that $u_1=\pm v_1$.
An identical argument holds for $(x-y) \tensor e$.
Thus, there exists $\alpha, \beta, \gamma, \delta \in \kk$ such that,
\begin{align*}
\phi( (x+y) \tensor e)_1 &= \alpha (x+y) \tensor e + \gamma (x+y) \tensor g \\
\phi( (x+y) \tensor g)_1 &= \gamma (x+y) \tensor e + \alpha (x+y) \tensor g \\
\phi( (x-y) \tensor e)_1 &= \beta (x-y) \tensor e + \delta (x-y) \tensor g \\
\phi( (x-y) \tensor g)_1 &= \delta (x-y) \tensor e + \beta (x-y) \tensor g.
\end{align*}
These elements generate the degree 1 component of $V_2$ and so
the following matrix must be nonsingluar,
\[ M = \begin{bmatrix}
	\alpha & 0 & \gamma & 0 \\
	\gamma & 0 & \alpha & 0 \\
	0 & \beta & 0 & \delta \\
	0 & \delta & 0 & \beta
	\end{bmatrix}.\]
But $\det(M) = -(\beta^2-\delta^2)(\alpha^2-\gamma^2)$, a contradiction
since the above argument gave us $\alpha=\pm \gamma$.
Note that we assumed above that we are in the case that
$\phi(1 \tensor g) = 1 \tensor g$ but the same argument works
in the case $\phi(1 \tensor g) = -1 \tensor g$.

Write
\[ \phi(x \tensor e) = \left(a(x+y) + b(x-y)\right) \tensor e +  \left(c(x+y) + d(x-y)\right) \tensor g.\]
Because $\phi$ is an isomorphism and the image of $x \tensor e$
determines the isomorphism, then $a \neq \pm c$ and $b \neq \pm d$.

\begin{theorem}
Let $\phi \in \Aut(V_2 \# \cS_2)$ and write 
$\phi(x \tensor e) = (ax + by) \tensor e + (cx + dy) \tensor g$
for $a,b,c,d \in \kk$.
The parameters satisfy one of the three following conditions:
\begin{itemize}
\item $a \in \kk^\times$, $b = c = d = 0$;
\item $b,d \in \kk$, $b \neq 0$, $b \neq -d$,
$a = -\frac{d^2}{b}$, $c = -d$;
\item $c,d \in \kk$, $c \neq -d$,
$a = \pm \sqrt{\frac{c^2+d^2}{2}}$, $b = \mp \sqrt{\frac{c^2+d^2}{2}}$.
 \end{itemize}
\end{theorem}

\begin{proof}
This is easily obtained by checking in Maple
which parameters satisfy the defining relation
and give a bijective map.
\end{proof}

\subsection*{Acknowledgment}
E. Kirkman was partially supported by
grant \#208314 from the Simons Foundation.

\bibliographystyle{plain}

\end{document}